\documentclass[10pt]{article}
\usepackage{geometry}                
\usepackage{graphicx}
\usepackage{fullpage}
\usepackage{amsmath}
\usepackage{amssymb}
\usepackage{amsthm}
\usepackage{url}
\usepackage{epsfig}
\usepackage{color}
\usepackage{refcount}
\usepackage{verbatim}
\usepackage{enumerate}
\usepackage{multirow}
\usepackage{framed}
\usepackage{slashbox}
\usepackage{caption}
\usepackage{subcaption}
\usepackage[normalem]{ulem}

\usepackage[square,numbers,sort&compress]{natbib}

\usepackage{xr}
\usepackage[breaklinks=true,colorlinks,citecolor=blue,linkcolor=blue]{hyperref}

\usepackage{calc}

\usepackage{tikz}
\usepackage{xifthen}
\usetikzlibrary{calc}
\usepackage{epstopdf}
\usetikzlibrary{positioning}
\usepackage{pgffor}

\usepackage{color}
\definecolor{clemson-orange}{RGB}{234,106,32}
\definecolor{chicago-maroon}{RGB}{128,0,0}
\definecolor{cincinnati-red}{RGB}{190,0,0}
\definecolor{soft-cyan}{RGB}{68,85,90}
\usepackage{fullpage}
\usepackage{multicol}

\usepackage[utf8]{inputenc}

\usepackage{array}
\newcolumntype{L}[1]{>{\raggedright\let\newline\\\arraybackslash\hspace{0pt}}m{#1}}
\newcolumntype{C}[1]{>{\centering\let\newline\\\arraybackslash\hspace{0pt}}m{#1}}
\newcolumntype{R}[1]{>{\raggedleft\let\newline\\\arraybackslash\hspace{0pt}}m{#1}}

\newcommand{\bb}{\mathbb}
\newcommand{\R}{\bb R}
\newcommand{\Z}{{\bb Z}}
\newcommand{\N}{{\bb N}}

\theoremstyle{definition}
\newtheorem{theorem}{Theorem}[section]
\newtheorem{lemma}[theorem]{Lemma}
\newtheorem{corollary}[theorem]{Corollary}
\newtheorem{prop}[theorem]{Proposition}

\newtheorem{definition}[theorem]{Definition}

\newtheorem{example}[theorem]{Example}

\newtheorem{conjecture}{Conjecture}

\usepackage{lineno}
\makeatletter
\newcommand*\patchAmsMathEnvironmentForLineno[1]{%
  \expandafter\let\csname old#1\expandafter\endcsname\csname #1\endcsname
  \expandafter\let\csname oldend#1\expandafter\endcsname\csname end#1\endcsname
  \renewenvironment{#1}%
     {\linenomath\csname old#1\endcsname}%
     {\csname oldend#1\endcsname\endlinenomath}}%
\newcommand*\patchBothAmsMathEnvironmentsForLineno[1]{%
  \patchAmsMathEnvironmentForLineno{#1}%
  \patchAmsMathEnvironmentForLineno{#1*}}%
\AtBeginDocument{%
\patchBothAmsMathEnvironmentsForLineno{equation}%
\patchBothAmsMathEnvironmentsForLineno{align}%
\patchBothAmsMathEnvironmentsForLineno{flalign}%
\patchBothAmsMathEnvironmentsForLineno{alignat}%
\patchBothAmsMathEnvironmentsForLineno{gather}%
\patchBothAmsMathEnvironmentsForLineno{multline}%
}

\DeclareMathOperator*{\lin}{lin}

\DeclareMathOperator*{\conv}{conv}

\DeclareMathOperator*{\intt}{int}

\numberwithin{equation}{section}  

\title{Two-halfspace closure}
\author{Amitabh Basu\thanks{Department of Applied Mathematics and Statistics, Johns Hopkins University, Baltimore, MD, USA ({\tt basu.amitabh@jhu.edu}, {\tt hjiang32@jhu.edu}). Supported by the ONR Grant N000141812096, NSF Grant CCF2006587 and the AFOSR Grant FA95502010341.} 
\and Hongyi Jiang\footnotemark[1]}
\date{\today}
\date{\today}
\begin{document}

\maketitle

\begin{abstract} We define a new cutting plane closure for pure integer programs called the {\em two-halfspace closure}. It is a natural generalization of the well-known Chv\'atal-Gomory closure. We prove that the two-halfspace closure is polyhedral. We also study the corresponding {\em 2-halfpsace rank} of any valid inequality and show that it is at most the split rank of the inequality. Moreover, while the split rank can be strictly larger than the two-halfspace rank, the split rank can be at most twice the two-halfspace rank. A key step of our analysis shows that the split closure of a rational polyhedron can be obtained by considering the split closures of all $k$-dimensional (rational) projections of the polyhedron, for any fixed $k \geq 2$. This result may be of independent interest.
\end{abstract}

\section{Cutting planes and closures}

A central question in discrete geometry and integer programming is understanding the convex hull of integer points in a polyhedron, both structurally and algorithmically. Motivated by Edmonds' early groundbreaking work in polyhedral combinatorics and combinatorial optimization, Chv\'atal~\cite{chvatal1973edmonds} proposed a very general method that provides insight into this question\footnote{While we  emphasize the perspective that Chv\'atal took, Gomory had developed a closely related approach in the late 50s and early 60s~\cite{Gomory63,gomory1960algorithm,MR0102437}.}. Shortcutting the historical development, the main idea is as follows. Let $P\subseteq \R^n$ be a rational polyhedron (allowing irrationality leads to some pathologies which we will avoid in this paper). Then for any rational halfspace $H$ containing $P$, we have $\conv(H\cap \Z^n) \supseteq \conv(P\cap \Z^n)$. It is easy to verify that $\conv(H \cap \Z^n)$ is again a halfspace called a {\it Chv\'atal-Gomory (CG) cutting plane for $P$}. It was shown by Chv\'atal~\cite{chvatal1973edmonds} that repeated applications of this operation can obtain any valid inequality for $\conv(P\cap \Z^n)$ (we will often use the notation $P_I$ to denote $\conv(P\cap \Z^n)$ for any polyhedron $P\subseteq \R^n$). To make this precise, for any polyhedron $P$ define the {\it Chv\'atal-Gomory (CG) closure of $P$} to be $$\mathcal{C}(P):=\bigcap_{\substack{H \textrm{ rational halfspace} \\ \textrm{such that }P\subseteq H}} H_I.$$ It is well-known that if $P\subseteq \R^n$ is a rational polyhedron, then $\mathcal{C}(P)$ is also a rational polyhedron even though it is defined by the intersection of infinitely many halfspaces (this is summarized below in Theorem~\ref{thm:chv-closure-II}). Thus, one can recursively define $\mathcal{C}^0(P):= P$ and $\mathcal{C}^k(P)= \mathcal{C}(\mathcal{C}^{k-1}(P))$ for $k \geq 1$. $\mathcal{C}^k(P)$  is called the {\it $k$-th Chv\'atal-Gomory (CG) closure} or the {\it rank $k$ Chv\'atal-Gomory (CG) closure}. The main power of this operation is summarized in the following theorem.

%
%
\begin{theorem}\label{thm:chv-closure-II}[\cite[Lemma 5.13]{conforti2014integer},\cite[Theorem 23.4]{sch}] There exists a computable function $t: \Z^{m\times n} \to \N$ such that for any rational polyhedron $P=\{x\in \R^n: Ax \leq b\}$ given by $A \in \Z^{m\times n}$ and $b\in \Z^m$, $$\mathcal{C}(P) = \bigcap_{\substack{ u \in [0,1)^n: \\ u^TA \in \Z^n}}\{x \in \R^n: \langle u^T A, x \rangle \leq \lfloor u^Tb \rfloor\}$$ and $$\mathcal{C}^{t(A)}(P) = \conv(P\cap \Z^n).$$ 
\end{theorem}

Thus, if one enumerates the finitely many points in $\{u \in [0,1)^n: u^TA \in \Z^n\}$, then one can compute the Chv\'atal-Gomory closure, and repeating this process $t(A)$ number of times, one can compute $\conv(P\cap \Z^n)$, which is called the {\it integer hull of $P$}.

\paragraph{Two-halfspace closure.} One can view the above discussion as follows. One wishes to compute $\conv(P\cap \Z^n)$ which seems complicated. However, if $P$ is a rational {\it halfspace}, then computing the integer hull is easy: simply express $H = \{x\in \R^n : \langle a, x \rangle \leq b\}$, where $a \in \Z^n$ has relatively prime coordinates (this can be done because $H$ is a rational halfspace) and then $H_I = \{x\in \R^n : \langle a, x \rangle \leq \lfloor b\rfloor \}$. One next observes that if one has a rational halfspace relaxation $H \supseteq P$, then clearly $H_I \supseteq P_I$. The deep insight from Theorem~\ref{thm:chv-closure-II} is that with a finite choice of these halfspace relaxations and a finite repetition of this operation, one can completely describe the integer hull of $P$.

Thus, the main idea is to find a ``simple" relaxation of $P$ whose integer hull is easier to compute. But then why stop at halfspace relaxations? What about relaxations of $P$ obtained by the intersection of {\it two} rational halfspaces and computing the integer hull of such relaxations? Surprisingly, to the best of our knowledge, this particular question has not been posed or investigated in the literature. In this paper, we initiate this discussion.

\begin{definition}
Given a polyhedron $P$, 
the {\it two-halfspace closure} of $P$, which is denoted by $\mathcal{T}(P)$, is defined as 
\begin{equation}
    \mathcal{T}(P)=\bigcap_{\substack{H_1,H_2 \mbox{ rational halfspaces}\\\mbox{such that }P\subseteq H_1\cap H_2}}(H_1\cap H_2)_I
\end{equation}
\end{definition}

To make this useful computationally, one revisits the following questions that come up for the Chv\'atal-Gomory cutting planes.  Affirmative answers are needed to make this definition interesting from an algorithmic perspective.

\begin{enumerate}
\item Is it easier to compute $(H_1 \cap H_2)_I$ compared to $P_I$ itself? 
\item Is $\mathcal{T}(P)$ a rational polyhedron when $P$ is a rational polyhedron?
\item If we take the two-halfspace closure repeatedly, then do we arrive at the convex hull in a finite number of steps?
\end{enumerate}

The answer to Question 1. above is YES, in the following sense. It is not hard to argue that computing $(H_1\cap H_2)_I$ is equivalent to computing a two-dimensional integer hull. This is done by projecting $H_1\cap H_2$ and $\Z^n$ onto the two-dimensional orthogonal complement of the lineality space of $H_1\cap H_2$ (see Proposition \ref{prop::IHtwo}). Computing a basis for the lattice obtained by projecting $\Z^n$ onto a linear subspace of $\R^n$ can be achieved in polynomial time using Hermite Normal Form computations~\cite[Chapters 4 and 5]{sch}. This reduces the problem to computing the integer hull of a simplicial cone in two dimensions. Since there are explicit polynomial time algorithms for computing two-dimensional integer hulls~\cite{harvey1999computing,cook-hartmann-kannan-mcdiarmid-1992}, this makes the problem of computing $(H_1\cap H_2)_I$ computationally easier than computing $P_I$, at least theoretically speaking.

The answer to Question 3. is YES, somewhat trivially: $\mathcal{T}(P) \subseteq \mathcal{C}(P)$ since one may take $H_1 = H_2$. Thus, by Theorem~\ref{thm:chv-closure-II}, a finite number of applications of the two-halfspace closure operation gives the convex hull; in fact, the number of steps needed for the two-halfspace closure can be much smaller compared to the Chv\'atal-Gomory procedure. We make this precise below in Corollary~\ref{cor:split-2-half}.

The main result of this paper is an affirmative answer to Question 2. above. 

\begin{theorem}\label{thm:2-halfspace-poly} Let $P$ be a rational polyhedron. Then $\mathcal{T}(P)$ is a rational polyhedron.
\end{theorem}

The idea of considering ``simple" relaxations of $P$ other than halfspaces and computing their integer hulls as a means to make progress towards $P_I$ is not new to this paper. Gomory~\cite{gom}, in an influential paper, considers relaxations of $P$ that are simplicial cones obtained from $n$ linearly independent defining constraints of $P$; he termed the integer hulls of these cones {\it corner polyhedra}. This idea has resulted in decades long research into corner polyhedra. More recently, {\it aggregation closures} were defined for covering and packing polyhedra~\cite{bodur2018aggregation}, with follow up work extending and generalizing this idea~\cite{pashkovich2019,lee2019generalized,lee2020generalized,dey2018strength}. In these works, a relaxation given by halfspaces is strengthened with variable bounds.

\paragraph{A comparison with the split closure.} The preceding discussion approaches the integer hull question by considering the integer hulls of ``simple" relaxations. Another approach to cutting planes comes from the idea of a {\it disjunction}, which is a finite collection of polyhedra whose union contains all of $\Z^n$. Given a disjunction $D = Q_1 \cup \ldots \cup Q_k$ such that $\Z^n \subseteq D$, a {\it cutting plane for $P$ derived from $D$} is defined to be any halfspace $H$ such that $P \cap D \subseteq H$. Since $\Z^n \subseteq D$, we have that $P\cap \Z^n \subseteq H$ and so $\conv(P\cap \Z^n) \subseteq H$. The idea again is that convexifying $P\cap D$ is easier than convexifying $P\cap \Z^n$. This is made more precise by appealing to Balas' work on concrete descriptions of the union of polyhedra and disjunctive programming~\cite{balas2018disjunctive}; see also~\cite[Sections 4.9 and 5.5]{conforti2014integer}.

Cook, Kannan and Schrijver~\cite{cook-kannan-schrijver} introduced and studied the simplest form of disjunctions: union of two disjoint halfspaces, i.e., disjunctions of the form $D_{a,K} = \{x\in \R^n: \langle a, x \rangle \leq K\}\cup  \{x\in \R^n: \langle a, x \rangle \geq K+1\}$, where $a\in \Z^n$ and $K \in \Z$. The closure of the complement of such disjunctions are called {\it split sets}, i.e., sets of the form $\{x\in \R^n: K \leq \langle a, x \rangle \leq K+1\}$. Define $P^{a,K}:= \conv(P\cap D_{a,K})$. The {\it split closure} of $P$ is defined to be \begin{equation}\label{eq:split-closure}\mathcal{S}(P) := \bigcap_{a \in \Z^n, K \in \Z} P^{a,K}.\end{equation}

In the spirit of the preceding discussions, a natural question is whether the split closure is a polyhedron. In~\cite{cook-kannan-schrijver}, the authors establish that this is indeed the case for any rational polyhedron $P$. Moreover, one can then repeat this operation and define the $k$-th split closure analogous to how the $k$-th Chv\'atal-Gomory closure was defined. The question as to whether a finite application of the split closure ends in the integer hull\footnote{For any family of disjunctions, one can pose similar questions about the closure with respect to this family; such discussions are outside the scope of this paper.} is also answered affirmatively by observing that $\mathcal{S}(P) \subseteq \mathcal{C}(P)$ because any Chv\'atal-Gomory cutting plane $\langle a, x \rangle \leq \delta$ is valid for $P^{a,\delta}$.

We wish to compare the strength of the two-halfspace closure and the split closure (recall that both are subsets of the Chv\'atal-Gomory closure). To make this precise, define $\mathcal{T}^k(P)$ to be the polyhedron obtained by $k$ repetitions of the two-halfspace closure operation and $\mbox{rank}_{\mathcal{T}}(P)$ to be the smallest natural number $k\in \N$ such that $\mathcal{T}^k(P) = P_I$. Similarly, define $\mathcal{S}^k(P)$ to be the polyhderon obtained by $k$ repetitions of the split closure operation and $\mbox{rank}_{\mathcal{S}}(P)$ to be the smallest natural number $k\in \N$ such that $\mathcal{S}^k(P) = P_I$. Then, our proof of Theorem~\ref{thm:2-halfspace-poly} gives the following result as a byproduct.

\begin{corollary}\label{cor:split-2-half} For all $k\geq 0$, $$\mathcal{S}^{2k}(P) \subseteq \mathcal{T}^k(P) \subseteq \mathcal{S}^k(P).$$ Consequently,
\end{corollary}
$$
    \frac{1}{2}\mbox{rank}_{\mathcal{S}}(P)\leq \mbox{rank}_{\mathcal{T}}(P)\leq \mbox{rank}_{\mathcal{S}}(P). 
$$

Moreover, in Section~\ref{sec:strict-cont} we give an example showing that the containment $\mathcal{T}(P) \subseteq \mathcal{S}(P)$ can be strict.

\paragraph{Multi-row cuts.} In the past 15 years or so, there has been a lot of research in the area of {\it multi-row cuts} where the general idea is similar to the philosophy of this paper. One wishes to derive valid inequalities for mixed-integer sets in the form $\min\{x \in \Z^n_+ \times \R^d_+: Ax = b\}$. Then one takes a relaxation by considering two or more aggregated constraints from $Ax = b$, along with the nonnegativity constraints, and attempts to construct the (mixed)-integer hull of these relaxations. This approach has a computational advantage because the aggregations are taken from the optimal simplex tableaux and the valid inequalities for these relaxations can be derived from simple formulas that exploit gauge and support function duality or use the Gomory-Johnson approach of subadditive functions; see~\cite{basu2015geometric} for a survey. Several polyhedrality results have been proven for this and related settings~\cite{dash2016polyhedrality,dash2019lattice,dash2017polyhedrality,MR2969261,bhk2}, but we do not see an immediate connection with the polyhedrality result of this paper.

\section{Proofs of Theorem~\ref{thm:2-halfspace-poly} and Corollary~\ref{cor:split-2-half}}

We begin with some preliminary definitions and simple observations.

\begin{definition} A linear subspace $L \subseteq \R^n$ is called a {\it lattice subspace} if $\conv(L\cap \Z^n) = L$. For any $X \subseteq L$, $\intt_L(X)$ will denote the interior of $X$ with respect to the relative topology of $L$. Given a lattice subspace $L$, $\Lambda_L$ will denote the lattice in $L$ obtained by the orthogonal projection of $\Z^n$ onto $L$. A subset of $L$ of the form $H \cap L$, where $H\subseteq \R^n$ is a halfspace, will be called a {\it halfspace in $L$} (note that this definition allows $L$ itself to be a halfspace in $L$). \end{definition}
\begin{definition} Let $L$ be a lattice subspace of $\R^n$. Let $P \subseteq L$ be a polyhedron. A halfspace $H$ in $L$ is a {\it Chv\'atal-Gomory (CG) cut for $P$ with respect to $\Lambda_L$} if there exists another halfspace $H'$ in $L$ such that $P\subseteq H'$ and $H = \conv(H' \cap \Lambda_L)$.


A {\it split set in $L$ with respect to $\Lambda_L$} is a subset $S \subseteq L$ such that $\intt_L(S) \cap \Lambda_L = \emptyset$, $S$ is the intersection of two halfspaces in $L$, and the dimension of $S$ is the same as the dimension of $L$. In other words,  it is a proper subset of $L$ that is the intersection of $L$ with some split set in $\R^n$. A halfspace $H$ in $L$ such that $P \setminus \intt_L(S) \subseteq H$ is called a {\it split cut for $P$ with respect to $\Lambda_L$}. \end{definition}

\begin{definition} For any polyhedron $P$ in  $\R^n$ and lattice subspace $L\subseteq \R^n$, 
$P_L$ will denote the image of $P$ orthogonally projected from $\R^n$ to $L$,  $\lin(P)$ will denote the lineality space of $P$, and $L^{\perp}$ will denote the orthogonal complement of $L$. Define 
\begin{align*}
    P_{I,L}&=\conv(P_L\cap \Lambda_L)\\
    \mathcal{S}_L(P) &= \bigcap_{\substack{\mbox{ split set }S \mbox{ on } L\\ \mbox{ with respect to }\Lambda_L}}\conv(P_L\backslash {\intt}_L(S)),\\
    \mathcal{C}_L(P) &= \bigcap_{\substack{\mbox{ CG cut }H \mbox{ for } P_L\\ \mbox{ with respect to }\Lambda_L}}H,
\end{align*}
 which will be called the {\it lattice hull, split closure} and {\it CG closure of $P$ in $L$ with respect to $\Lambda_L$}, respectively. With a little bit of abuse of notation, we let $P_I=P_{I,\R^n}$, $\mathcal{S}(P)$ denote $\mathcal{S}_{\R^n}(P)$, and $\mathcal{C}(P)$ denote $\mathcal{C}_{\R^n}(P)$. 
\end{definition}



\begin{prop}\label{prop::IHtwo}
Let $H_1, H_2$ be two rational halfspaces in $\R^n$ such that $H_1 \cap H_2$ is not $\emptyset$ or $\R^n$. Let $L$ be the linear subspace $\lin(H_1\cap H_2)^\perp$; thus $L$ is two dimensional if $H_1$ and $H_2$ are not parallel and $L$ is one-dimensional otherwise. Then $(H_1\cap H_2)_I=(H_1\cap H_2)_{I,L}+L^\perp = (H_1\cap H_2)_{I,L}+\lin(H_1 \cap H_2)$.
\end{prop}
\begin{proof}
For $z\in (H_1\cap H_2)\cap \Z^n$, by definition $z\in(H_1\cap H_2)_{I,L}+L^\perp $, and so $(H_1\cap H_2)_I\subseteq (H_1\cap H_2)_{I,L}+L^\perp$. 

For $z'\in (H_1\cap H_2)_{L}\cap \Lambda_L$, there exists $z_1\in \Z^n$, s.t. $z'$ is the projection of $z_1$. In other words, there exists $v_1 \in L^\perp$ such that $z' = z_1 + v_1$. Similarly, there exist $z_2 \in H_1 \cap H_2$ and $v_2 \in L^\perp$ such that $z ' = z_2 + v_2$. Therefore, $z_1 = z_2 + v_2 - v_1$ and so $z_1 \in H_1 \cap H_2$ as well. This shows that $z_1 \in (H_1 \cap H_2) \cap \Z^n$. Finally, observe that $z' + L^\perp = z_1 + L^\perp$ since $z' - z_1 \in L^\perp$. 

Since $H_1, H_2$ are both rational halfspaces, $(H_1 \cap H_2)_I$ has the same lineality space $L^\perp$ as $H_1\cap H_2$. Therefore, since $z_1 \in (H_1 \cap H_2) \cap \Z^n$, we must have $z_1 + L^\perp \subseteq (H_1 \cap H_2)_I$ and therefore $z' + L^\perp \subseteq (H_1 \cap H_2)_I$. 

Thus, we conclude that for any $z' \in (H_1\cap H_2)_{L}\cap \Lambda_L$, $z' + L^\perp \subseteq (H_1 \cap H_2)_I$. Since $(H_1 \cap H_2)_I$ is convex and has lineality space $L^\perp$, this shows that $(H_1\cap H_2)_{I,L}+L^\perp \subseteq (H_1 \cap H_2)_I$.\qed
%
%
%
\end{proof}

\begin{prop}\label{prop::pull-back}
Let $P\subseteq \R^n$ be a polyhedron and let $L$ be a lattice subspace of $\R^n$. Then the following are all true.
\begin{enumerate}
   \item Let $H$ be a halfspace in $L$. Then $P_L\subseteq H$ implies $P\subseteq H+L^\perp$.\label{prop::pull-back-valid}
   
    \item Let $H$ be a halfspace in $L$. Then if $P_L\cap \Lambda_L\subseteq H$ we have $P\cap \Z^n\subseteq H+L^\perp$. \label{prop::pull-back-cut}
    
     \item Let $S$ be a split set in $L$ with respect to $\Lambda_L$, then $S+L^\perp$ is a split set in $\R^n$. Consequently, if $H$ is a split cut for $P_L$ with respect to $\Lambda_L$ then $H + L^\perp$ is a split cut for $P$. \label{prop::pull-back-split}
     
    \item If $H$ is a CG cut for $P_L$ with respect to $\Lambda_L$, then $H+L^\perp$ is a CG cut for $P$ with respect to $\Z^n$. \label{prop::pull-back-CG}
\end{enumerate}
\end{prop}

\begin{proof}
\begin{enumerate}
 \item This is because $P\subseteq (P_L+L^\perp)\subseteq  H+L^\perp$.
  \item Given a halfspace $H$ in $L$, suppose $P_L\cap \Lambda_L \subseteq H$. Since $P\cap \Z^n$ projects to $P_L\cap \Lambda_L $, we have $P\cap \Z^n \subseteq (P_L \cap \Lambda_L) + L^\perp \subseteq H+L^\perp$.
    
    \item For an integer point $z\in \Z^n$, it is orthogonally projected onto some lattice point $z'\in \Lambda_L$. Then $z'\notin S$ implies that $(z'+L^\perp)\cap (S+L^\perp) = \emptyset$. Since $z\in z'+L^\perp$, so we obtain $z\notin S+L^\perp$, and thus $S+L^\perp$ is a split set (we use the fact that for any halfspace $H \subseteq L$, $H+ L^\perp$ is a halfspace in $\R^n$). Then by part \ref{prop::pull-back-valid}, we can prove the rest of part \ref{prop::pull-back-split}.
   
    \item Since $H$ is a CG cut for $P_L$ with respect to $\Lambda_L$, we have $P\cap \Z^n\subseteq H+L^\perp$ by part \ref{prop::pull-back-cut}. Let $D^{a,K}$ be the split disjunction in $L$ that derives the CG cut $H$. By part~\ref{prop::pull-back-split}, $D^{a,K}+L^\perp$ is a split disjunction in $\R^n$. Let $H_1$ and $H_2$ be the two halfspaces  in $L$ such that $D^{a,K}=H_1\cup H_2$. Since $H$ is a CG cut, we can assume $P_L\cap \Lambda_L\subseteq H_1=H$, and $P_L\cap H_2=\emptyset$. Then $P\cap \Z^n\subseteq H_1+L^\perp=H+L^\perp$ by part \ref{prop::pull-back-cut}, and $P \cap (H_2 + L^\perp) = \emptyset$ by part \ref{prop::pull-back-valid}, which finishes the proof.\qed \end{enumerate}\end{proof}

\begin{definition} For $k \in \{1, \ldots, n\}$, let $\mathcal{L}_k$ denote the set of all $k$ dimensional lattice subspaces in $\R^n$.
\end{definition}

The next result says that the two-halfspace closure can be obtained by intersecting all Chv\'atal-Gomory cuts for the split closures of the two-dimensional projections of a polyhedron. This essentially follows from the fact that the integer hull of a two-dimensional simplicial cone can be obtained by taking the split closure of the simplicial cone, and then taking the Chv\'atal-Gomory closure. This fact was first observed in~\cite{basu-split-2D} but we include a proof for completeness in the Appendix; see Theorem~\ref{thm:rank-IH}.

\begin{lemma}\label{thm::TP-decomp}
Given a polyhedron $P\subseteq \R^n$, for $L\in \mathcal{L}_2$, we define $$\mathcal{K}(P,L)=\{H \mbox{ is a halfspace in } \R^n: P\subseteq H,~ L^\perp\subseteq \lin(H)\}.$$ 
For $H_1,H_2\in \mathcal{K}(P,L)$, we define
\begin{align}
    \mathcal{G}_{L,H_1,H_2}=&\{H \subseteq \R^n: H=H'+L^\perp \nonumber\\&\mbox{ for } H' \mbox{ which is a CG cut for } \mathcal{S}_L(H_1\cap H_2) \mbox{ in } L\}.
\end{align}
Let $\mathcal{G}_L=\bigcup_{H_1,H_2\in\mathcal{K}(P,L) }\mathcal{G}_{L,H_1,H_2}$. 
Let $\mathcal{G} = \cup_{L\in \mathcal{L}_2} \mathcal{G}_L$. Let $P^1:=\bigcap_{L\in\mathcal{L}_2}(\mathcal{S}_L(P)+L^\perp)$. Then we have 
\begin{equation}
    \mathcal{T}(P)=\bigcap_{H\in \mathcal{G}}P^1\cap H.\label{eq::add-CG}
\end{equation}
Moreover, each $H\in \mathcal{G}$ is a CG cut for $P^1$.
\end{lemma}

\begin{proof}
We have 
\begin{align}
    \mathcal{T}(P)&=\bigcap_{L\in\mathcal{L}_2}\bigcap_{~H_1,H_2\in\mathcal{K}(P,L) }   \big((H_1\cap H_2)_{I,L}+L^\perp \big)\label{eq::all-IHtwo}\\
    &=\bigcap_{L\in\mathcal{L}_2}\bigcap_{~H_1,H_2\in\mathcal{K}(P,L) }   \big(\mathcal{C}_L(\mathcal{S}_L(H_1\cap H_2))+L^\perp \big)\label{eq::2D-rank}\\
    &=\bigcap_{L\in\mathcal{L}_2}\bigcap_{~H_1,H_2\in\mathcal{K}(P,L) }   \left((\mathcal{S}_L(H_1\cap H_2)+L^\perp \big)\bigcap_{H\in \mathcal{G}_{L,H_1,H_2}}H\right)\label{eq:def-G}\\
    &=\bigcap_{L\in\mathcal{L}_2}  \left(\left(\left(\bigcap_{~H_1,H_2\in\mathcal{K}(P,L) } \mathcal{S}_L(H_1\cap H_2)\right)+L^\perp \right)\right.\nonumber\\
    &~~\left.\bigcap_{~H_1,H_2\in\mathcal{K}(P,L) }\bigcap_{H\in \mathcal{G}_{L,H_1,H_2}}H\right)\label{eq:dist-cap}\\
    &=\bigcap_{L\in\mathcal{L}_2}  \left(\left( \mathcal{S}_L(P)+L^\perp \right)\bigcap_{H\in \mathcal{G}_{L}}H\right)\label{eq::intersection-split}\\
    &=\bigcap_{H\in \mathcal{G}}P^1\cap H\label{eq:final}
\end{align}


Equation (\ref{eq::all-IHtwo}) is due to Proposition \ref{prop::IHtwo}.

Equation (\ref{eq::2D-rank}) is based on the fact that the integer hull of a simplicial cone $Q\in \R^2$ can be derived by taking split closure of $Q$, and then taking the CG closure of this split closure~\cite{basu-split-2D}. We include a proof of this fact for completeness in Theorem~\ref{thm:rank-IH} in the Appendix.

Equation~\eqref{eq:def-G} follows from the definition of $\mathcal{G}_{L,H_1,H_2}$ and Equation~\eqref{eq:dist-cap} simply distributes the intersection operator.

For equation (\ref{eq::intersection-split}), consider $L\in \mathcal{L}_2$, let $H_1'$ and $H_2'$ be two halfspaces in $L$ that define a simplicial cone containing $P_L$. Then by  Proposition \ref{prop::pull-back},  $H_1'+L^\perp, H_2'+L^\perp\in \mathcal{K}(P,L)$. It is well known that the intersection of the split closures of all the simplicial cones containing a rational polyhedron $Q$ is the split closure of $Q$ (see \cite{andersen2005split}). Thus, we have $$\bigcap_{~H_1,H_2\in \mathcal{K}(P,L)} \mathcal{S}_L(H_1\cap H_2)=\mathcal{S}_L(P),$$ which finally leads to Equation (\ref{eq::intersection-split}). Distributing the intersection over $L \in \mathcal{L}_2$ yields equation~\eqref{eq:final}.

For each $H\in \mathcal{G}$, by definition, there exists $L\in \mathcal{L}_2$ and halfspaces $H_1, H_2\in \mathcal{K}(P,L)$, such that $H=H'+L^\perp$, where $H'$ is a CG cut for $\mathcal{S}_L(H_1\cap H_2)$. By Proposition \ref{prop::pull-back}, part~\ref{prop::pull-back-CG}., $H$ is a CG cut for $\mathcal{S}_L(H_1\cap H_2)+L^\perp$. Since $P^1\subseteq \mathcal{S}_L(H_1\cap H_2)+L^\perp$, we have $H$ is a CG cut for $P^1$, which finishes the proof.\qed 
\end{proof}



Now we will show that $P^1=\mathcal{S}(P)$. In fact, we will show a more general theorem, which says that one can obtain the split closure of a polyhedron $P$ by considering the split closures of all $k$-dimensional projections of $P$, for any $k \geq 2$.
\begin{theorem}\label{thm:split-projs} Fix any $k \in \{2, \ldots, n\}$.
For any polyhedron $P\in \R^n$, we have 

\begin{equation}
    \mathcal{S}(P)=\bigcap_{L\in\mathcal{L}_k}(\mathcal{S}_L(P)+L^\perp)
\end{equation}
\end{theorem}

{By Proposition \ref{prop::pull-back} part~\ref{prop::pull-back-split}., if $H$ is a split cut for $P_L$ with respect to $\Lambda_L$ for some $L\in \mathcal{L}_k$, then $H+L^\perp$ is a split cut for $P$. This means $\mathcal{S}(P)\subseteq \mathcal{S}_L(P)+L^\perp$. Therefore, $\mathcal{S}(P)\subseteq \bigcap_{L\in\mathcal{L}_k}(\mathcal{S}_L(P)+L^\perp)$. The other direction is established in the next theorem.}
\begin{theorem}\label{thm:proj} Fix any $k \in \{2, \ldots, n\}$. Let $P\in \R^n$ be a polyhedron, $S$ be a split set in $\R^n$ which is described by $\delta_1 \leq \langle a_1,x\rangle\leq \delta_1+1$, where $a_1\in \Z^n$ and $\delta_1\in \Z$.  Assume $H\in \R^n$ is a halfspace represented by $\langle a_2,x \rangle\leq \delta_2 $ such that $H\supseteq P\backslash \intt(S)$.  { Let $L\in \mathcal{L}_k$ be any $k$ dimensional lattice subspace containing $a_1$ and $a_2$ (such an $L$ always exists because $k \geq 2$). Then the following are both true, where $T(\cdot)$ is the orthogonal projection operator from $\R^n$ to $L$.}
\begin{enumerate} \item $T(H)+L^\perp=H$, and $T(H) = H \cap L$ and is therefore a halfspace in $L$. \item $T(S)$ is a split set in $L$ with respect to $\Lambda_L$ and  $T(H)\supseteq T(P) \backslash \intt_L(T(S))$. In other words, $T(H)$ is a valid halfspace in $L$ for $\mathcal{S}_L(P)$.\end{enumerate}\end{theorem}

\begin{proof} {For any convex set $C \subseteq \R^n$ and any {linear} subspace $V \subseteq \R^n$ such that $\lin(C)^\perp \subseteq V$, {we will prove that} $C = (C\cap V) + V^\perp$. Since $C\cap V \subseteq C$ and $V^\perp \subseteq \lin(C)$, we have $(C \cap V) + V^\perp \subseteq C + \lin(C) = C$. Consider any $x \in C$ and so $x + V^\perp \subseteq x + \lin(C)\subseteq C$. Let $y \in (x + V^\perp)\cap V$ and therefore $y \in C \cap V$. Since $V^\perp$ is a linear subspace, $x \in y + V^\perp$, so we get $x \in (C\cap V) + V^\perp$. This also shows that the orthogonal projection of $C$ {onto} $V$ is simply $C\cap V$.

Note that $\lin(H)^\perp$ is the line spanned by $a_2$ which is contained in $L$. So the above observations can be applied with $C = H$ and $V = L$, giving us $H = (H \cap L) + L^\perp = T(H) + L^\perp $. This establishes 1. 

Let $H_1 = \{x \in \R^n : \langle a_1, x \rangle \geq \delta_1\}$ and $H_2= \{x \in \R^n : \langle a_1, x \rangle \leq \delta_1 + 1\}$ be the halfspaces. $\lin(H_1)^\perp = \lin(H_2)^\perp$ is the line spanned by $a_1$ which is contained in $L$. Applying the observation with $C = H_1, H_2$ and $V = L$, we obtain $T(H_i) = H_i \cap L $ for $i=1,2$. Thus, $T(H_i)$ are halfspaces in $L$. Applying the observation to $C=S$ and $V=L$, we obtain $T(S) = S \cap L = H_1 \cap H_2 \cap L = (H_1 \cap L) \cap (H_2 \cap L) = T(H_1) \cap T(H_2)$. Thus, $T(S)$ is the intersection of two halfspaces in $L$. A similar argument as above shows that $\intt(S) = (\intt(S) \cap L) + L^\perp = \intt_L(S \cap L) + L^\perp = \intt_L(T(S)) + L^\perp$. Since $\Z^n \cap \intt(S) = \emptyset$, we have that $\Z^n \cap (\intt_L(T(S)) + L^\perp) = \emptyset$. This implies $\Lambda_L \cap \intt_L(T(S)) =\emptyset$ showing that $T(S)$ is a split set in $L$ with respect to $\Lambda_L$.

{We finally check that $T(H)\supseteq T(P)\backslash \intt_L(T(S))$. For $x\in T(P)\backslash \intt_L(T(S))$, there exists $v\in L^\perp$ such that $T(x + v)=x$ and $x+v \in P$. Since $x\notin \intt_L(T(S))$, and $\intt(S)=\intt_L(T(S)) + L^\perp$ as proved above,  we have $x + v\notin \intt(S)$, and thus $x + v\in P\backslash \intt(S)\subseteq H$. Thus $x\in T(H)$. 
}
}\qed

\end{proof}

By taking $k=2$ in Theorem~\ref{thm:split-projs}, we obtain

\begin{corollary}\label{cor::P1-is-Psplit}
Given a polyhedron $P$ in $\R^n$, we have $P^{1}= \mathcal{S}(P)$, where $P^1$ is as defined in Lemma~\ref{thm::TP-decomp}.
\end{corollary}

We have finally collected all the tools to prove Theorem~\ref{thm:2-halfspace-poly} and Corollary~\ref{cor:split-2-half}.

\begin{proof}[Proof of Theorem~\ref{thm:2-halfspace-poly}] By Corollary~\ref{cor::P1-is-Psplit} and the fact that the split closure is a polyhedron~\cite{cook-kannan-schrijver}, we have that $P^1$ is a polyhedron. By Lemma~\ref{thm::TP-decomp}, $\mathcal{T}(P)$ is obtained from $P^1$ by adding a subset of Chv\'atal-Gomory cuts for $P^1$. By Theorem 1.1 in \cite{MR2969261}, such a subset of CG cuts is dominated by a finite subset of CG cuts\footnote{In fact, Theorem 1.1 in \cite{MR2969261} shows that any collection of CG cuts contains a finite subcollection of cuts that dominates the entire collection. { This is also proved in \cite{andersen2005split}.}}. This completes the proof.\qed
\end{proof}

\begin{proof}[Proof of Corollary~\ref{cor:split-2-half}] One observes that $\mathcal{S}^2(P) \subseteq \mathcal{T}(P)\subseteq \mathcal{S}(P)$ by Lemma \ref{thm::TP-decomp}, Corollary \ref{cor::P1-is-Psplit}, and the fact that any CG cut is a split cut. Applying this observation iteratively proves the corollary.\qed
\end{proof}

\section{Example that shows the containment $\mathcal{T}(P)\subseteq \mathcal{S}(P)$ can be strict}\label{sec:strict-cont}
\begin{example}\label{exp::strict-contain}
Consider a simplicial cone $P=H_1\cap H_2\subseteq \R^2$, where  
$H_1=\{(x_1,x_2)\in \R^2:x_2\leq 2x_1+1/2\}$ and $H_2=\{(x_1,x_2)\in \R^2:x_2\leq -2x_1+5/2\}$. 
\end{example}

 \noindent{\bf Claim: }Then the point  $z:=\left(\frac{1}{2},\frac{1}{99}\right)$ is in $\mathcal{S}(P)$ but not in $\mathcal{T}(P)$.
\begin{figure}
\centering
\includegraphics[width=0.5\textwidth]{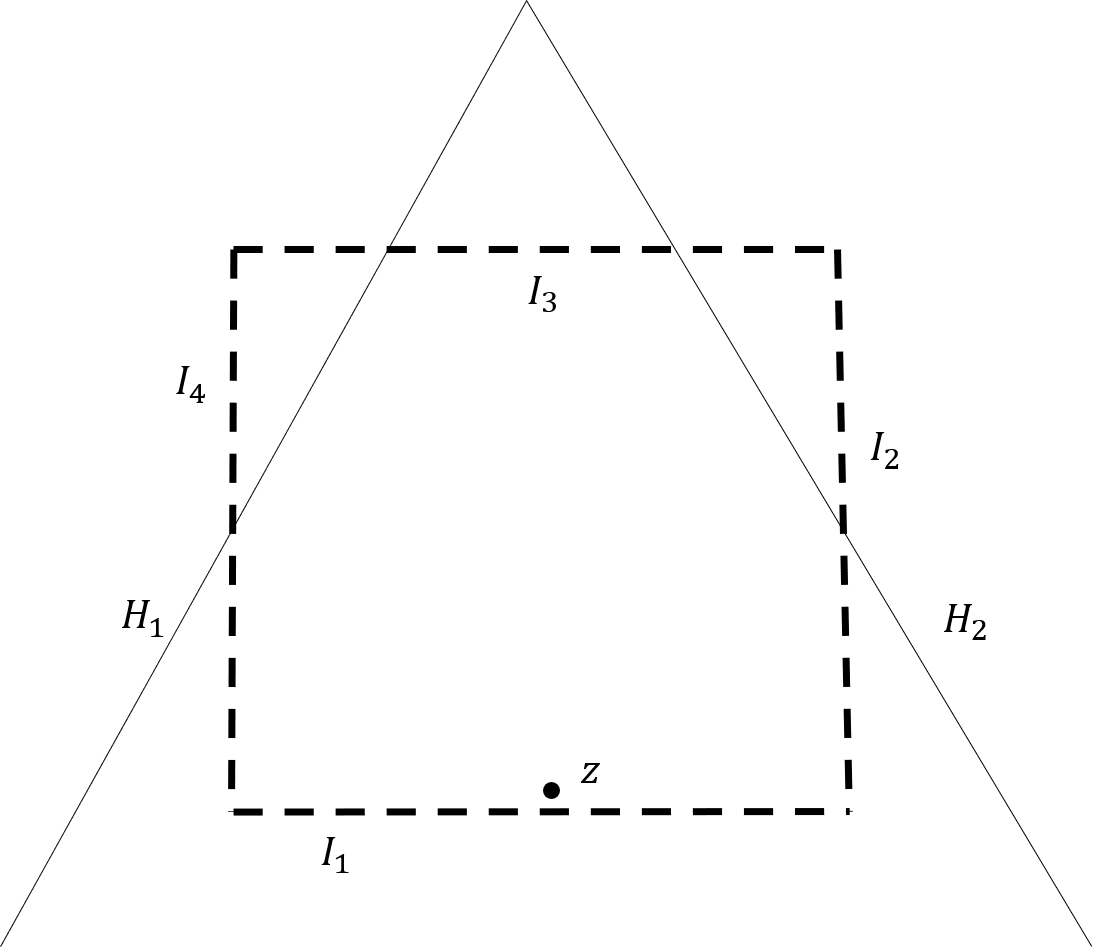}
\caption{Example \ref{exp::strict-contain}}

\end{figure}
\begin{proof}
It is clear that $\mathcal{T}(P)=P_I$, so $z\notin \mathcal{T}(P)$ since $P_I\subseteq\{(x_1,x_2)\in \R^2:x_2\leq 0\}$. 

Let $I_1,I_2,I_3$, and $I_4$ denote four segments: $\conv(\{(0,0), (1,0)\})$, $\conv(\{(1,0),(1,1)\})$, $\conv(\{(1,1),(0,1)\})$, and $\conv(\{(0,1),(0,0)\})$. Then for any split set $S$ such that $z\in S$, $S$ has to intersect with $\intt(\conv(I_1\cup I_2\cup I_3\cup I_4)$. Thus the split set has to intersect with the relative interior of exactly two of $I_1,\ldots, I_4$.  To produce a split cut that cuts off the apex of $P$, it can only {intersect} with the relative interior of $I_1$ and $I_3$, $I_2$ and $I_3$, or $I_3$ and $I_4$. Also, since $z\notin \conv(I_3\cup I_4)\cup \conv(I_3\cup I_2)$, so any split set that intersects with  interior of $I_2$ and $I_3$, or $I_3$ and $I_4$, it does not contain $z$. Thus we only need to consider the split sets that intersect with the relative interior of $I_1$ and $I_3$. It is clear that any split cut produced by such split sets is valid to $P^{(1,0), 0}$, using the notation from~\eqref{eq:split-closure}.
Then by simple calculation, $z\in P^{(1,0), 0}$. Thus $z\in \mathcal{S}(P)$.\qed
\end{proof}

\section{Future directions} One can naturally define the {\it $k$-halfspace closure} for any fixed natural number $k$: one considers a polyhedral relaxation $Q$ of $P$ defined by the intersection of $k$ halfspaces and then valid inequalities for $Q_I$ can be used as cutting planes for $P$. The closure is then defined as $$\mathcal{H}_k(P):=\bigcap_{\substack{k \textrm{-halfspace rational }\\ \textrm{relaxation }Q\textrm{ such that }\\P\subseteq Q}} Q_I.$$ For a fixed natural number $k$, computing the integer hull of a polyhedron $Q$ with at most $k$ facets is again a $k$-dimensional integer hull question by projecting onto the orthogonal complement of the lineality space of $Q$, { and, for example, enumerating integer points as in \cite{cook-hartmann-kannan-mcdiarmid-1992} and convexifying it, which is polynomially computable}. Hence, when $n$ is much larger than $k$, this is a reasonable operation and from a theoretical perspective one can let $n$ grow, but keep $k$ fixed and try to understand $\mathcal{H}_k(P)$. One is naturally led to

\begin{conjecture} For any fixed natural number $k\in \N$, and any rational polyhedron $P$, $\mathcal{H}_k(P)$ is a rational polyhedron.
\end{conjecture}

A well-known result~\cite{bell1974intersections,bell1975improved} in integer programming theory says that for any two dimensional rational polytope $P$, the integer hull $P_I$ is given by the intersection of all integer hulls of two-halfspace relaxations of $P$ obtained from every pair of facets of $P$ (for completeness, we include an alternate proof of this result in Appendix~\ref{sec:2DIH}). This result implies that the two-halfspace closure $\mathcal{H}_2(P) = \mathcal{T}(P)$ for any rational polyhedron $P\subseteq \R^n$ is the same as considering all possible two-dimensional rational projections of $P$, taking the two-dimensional integer (lattice) hull and ``lifting back" to the original space, and intersecting over all possible two dimensional projections. One can also generalize this idea by considering all possible $k$-dimensional projections: define $$\mathcal{P}_k(P) := \bigcap_{L\in\mathcal{L}_k} \big(P_{I,L}+L^\perp\big).$$ It is not hard to see that $\mathcal{C}(P) = \mathcal{P}_1(P){=\mathcal{H}_1(P)}$. Thus, this can be considered a natural generalization of the Chv\'atal-Gomory closure as well.
\medskip

For $k\geq 3$, it is not clear if $\mathcal{H}_k(P) = \mathcal{P}_k(P)$, like the case of $k=1,2$. We strongly suspect this is not the case, because the $k=2$ case crucially uses the fact that for a polyhedron in two-dimensions, the integer hull is precisely the the intersection of  integer hulls of all two-halfspace relaxations, and one may restrict one's attention to two-halfspace relaxations coming from the defining inequalities of the polyhedron itself (see Theorem~\ref{thm:2DIH} in Appendix \ref{sec:2DIH}). Firstly, it is not clear if one can restrict one's attention to corner polyhedra in the projection to describe $\mathcal{H}_k(P)$, for three and higher dimensions. Secondly, it is known that, in general, the intersection of all corner polyhedra does {\it not} give the integer hull for dimensions three and higher~\cite{bell1974intersections,bell1975improved}. Thus, we are led to believe

\begin{conjecture} For any fixed natural number $k\geq 3$, and any $n \geq k$, there are instances of rational polyhedra $P\subseteq \R^n$ such that $\mathcal{H}_k(P) \neq \mathcal{P}_k(P)$.
\end{conjecture}

{The discussion above shows that Conjecture 2 is true for $n=k$ since $\mathcal{P}_k(P)$ is simply the integer hull of $P$.} Finally, one wonders if $\mathcal{P}_k(P)$ is polyhedral for any fixed $k$.

\begin{conjecture} For any fixed natural number $k\in \N$, and any rational polyhedron $P$, $\mathcal{P}_k(P)$ is a rational polyhedron.
\end{conjecture}

We feel these questions in discrete geometry are worth pursuing in the future. { Finally, we would like to mention that the Ph.D. thesis of Wolfgang Keller addresses similar issues and defines various closely related cutting plane closures~\cite{keller2019tightening}. This thesis and the closures defined in this paper open up a number of questions regarding the relationship and relative strengths of these various cutting plane strategies.}


\section*{Acknowledgement}
We are very grateful to two anonymous referees for very insightful comments. Their suggestions and pointers helped to improve the paper from its initial versions. In particular, one of the referees suggested a shorter and more elegant proof for Theorem~\ref{thm:proj} which we adopted.

\bibliographystyle{plain}
\bibliography{full-bib}

\appendix

\section{Integer hull of two dimensional simplicial cones}


\begin{definition}\label{def:ell}
Given a simplicial cone $C{\subseteq} \R^2$, let $F^1,\dots, F^n$ be the facets of $C_I$. For every $i\in\{1,\dots,n\}$ we denote by $H^i$ the halfspace defining $F^i$, which can be described as $\{x\in \R^2:\langle a_i,x \rangle\leq \delta_i\}$, for some $a_i\in \Z^2$ of which the two entries are coprime and $\delta_i\in \Z$. Furthermore, we define $\widehat{\ell^i}$ as $\{x\in \R^2:\langle a_i,x \rangle= \delta_i+1\}$.
Also, let $\widehat{H^i}$ denote the halfspace described by $\{x\in \R^2:\langle a_i,x \rangle\leq \delta_i+1\}$.
\end{definition}


\begin{definition}\label{def:unit}
Given a line {$\ell\subseteq \R^2$} containing integer points, we call each closed segment whose endpoints are two consecutive integer points of $\ell$ as a {\it unit interval of $\ell$}.
\end{definition}
\begin{theorem}\label{thm:rank-IH}
 Given a simplicial cone $C\in \R^2$, the integer hull of  $C$ can be derived by taking the split closure of $C$, and then taking the CG closure of this split closure. 
\end{theorem}

\begin{proof}

We use the same notations as in Definition \ref{def:ell}. We first verify that $\mathcal{S}(C)\subseteq  \widehat{H^i}$ for every $i=1, \ldots, n$. Given $i\in \{1,\ldots,n\}$, if $\widehat{l^i}$ does not intersect with $C$, then $C$ is contained in $\widehat{H^i}$ and we are done. If $\widehat{l^i}$ intersects with $C$. Then there exists a unique unit interval $U$ of $\widehat{l^i}$ intersecting $C$. By definition of $\widehat{l^i}$, the two integer points of $U$ are outside of $C$. Also, there exists at least one unit interval on $F^i$. Let $U'$ be one of them. Then we take the unique split set $S$ containing the apex of $C$ and $\conv(U\cup U')$. The two intersection points between the boundaries of $S$ and the boundaries of $C$ determine a split cut $H$ produced by $S$, which cuts off $U\cap C$. Thus $\mathcal{S}(C)\subseteq H\cap C\subseteq \widehat{H^i}{\backslash \widehat{l^i}}$.

Thus by the definition of $\widehat{H^i}$, the facet defining halfspace $H^i$ of $F^i$ is a CG cut for $\mathcal{S}(C)$, for $i\in \{1,\ldots,n\}$. This finishes the proof.\qed
\end{proof}

\section{General integer hulls in two dimensions}\label{sec:2DIH}

We give a new proof of the following result.

\begin{theorem}\label{thm:2DIH}\cite{bell1974intersections,bell1975improved}
Let $P\subseteq \R^2$ be any two-dimensional rational polyhedron. Then $P_I$ is equal to the intersection of the integer hulls of all the simplicial cones (or split sets) constructed from pairs of facets of $P$.
\end{theorem}

We break the proof into two cases: $P_I$ is nonempty or empty. These are dealt with in Theorems~\ref{thm::non-empty-IH} and~\ref{thm::empty-IH} below.

\begin{definition}
A convex set $B\subseteq \R^n$ is a lattice-free convex set if there is no integer point in its interior, and it is maximal if for any lattice-free set $B'\supseteq B$, we have $B'=B$. 
\end{definition}

\begin{theorem}\label{thm::lat-free-2D}
In $\R^2$, a convex set $B$ is a maximal lattice-free convex set if and only if it satisfies one of the following properties. 

\begin{enumerate}
    \item $B$ is a split set, and each of its facets contains integer points.
    \item $B$ is a triangle such that each of its facets contains at least one integer point in its relative interior. 
    \item $B$ is a four-facet lattice-free set and each of its facet contains exactly one integer point in its relative interior. Moreover, the convex hull of the union of these four integer points is a parallelogram with area $1$.
\end{enumerate}
\end{theorem}

\begin{lemma}\label{lem::valid-ineq}
Let $Q\in \R^2$ be a rational polyhedron such that $\intt(Q)\cap \Z^2=\emptyset$, and $Q$ has four facets $F_1$, $F_2$, $F_3$ and $F_4$ in clockwise order. Let $H_i$ and $l_i$ denote the corresponding facet-defining halfspace and hyperplane for $i=1,2,3,4$. Assume $F_1$ contains at least one integer point in its relative interior, and $l_i\cap \Z^2=\emptyset$ for $i=2,3,4$. Let the corresponding facet defining halfspace $H_i$ be $\{x\in\R^2:\langle a_i,x\rangle\leq \delta_i\}$ for $i=1,2,3,4$. Then  $H:=\{x\in\R^2:\langle a_1,x\rangle\geq \delta_1\}\supseteq (H_2\cap H_3)_I\cap (H_2\cap H_4)_I\cap (H_3\cap H_4)_I$.  
\end{lemma}

\begin{proof}
We will attempt to construct a maximal lattice-free set containing $Q$ by ``pushing its facets out". More formally, we do the following.

If $(H_1 \cap H_2 \cap H_4) \cap \Z^2 = (H_1 \cap H_2 \cap H_3 \cap H_4) \cap \Z^2$, i.e., removing $H_3$ does not change the set of integer points, then $H \supseteq (H_2\cap H_4)_I$ and we are done. Suppose then that $(H_1 \cap H_2 \cap H_4) \cap \Z^2$ contains integer points that are not in $(H_1 \cap H_2 \cap H_3 \cap H_4) \cap \Z^2$. All such integer points must be in the interior of $H_1 \cap H_2 \cap H_4$, since $\ell_2$ and $\ell_4$ do not contain integer points. Therefore, there exists $\delta'_3 > \delta_3$ such that if one defines $H'_3 = \{x \in \R^2: \langle a_3, x \rangle \leq \delta'_3\}$, then $H_1 \cap H_2 \cap H'_3 \cap H_4$ is also lattice free, but the facet corresponding to $H'_3$ contains integer points in its relative interior. One now checks if removing $H_2$ introduces new integer points in $H_1 \cap H_2 \cap H'_3 \cap H_4$. If not, then we observe that $H \supseteq (H_3\cap H_4)_I$ since $\delta'_3 > \delta_3$ and we are done. Otherwise, we find $\delta'_2 > \delta_2$ such that if one defines $H'_2 = \{x \in \R^2: \langle a_2, x \rangle \leq \delta'_2\}$, then $H_1 \cap H'_2 \cap H'_3 \cap H_4$ is also lattice free, but the facet corresponding to $H'_2$ contains integer points in its relative interior. Finally, we ``push out" $H_4$ and either realize that $H \supseteq (H_2\cap H_3)_I$, or end up with a maximal lattice-free quadrilateral satisfying Case 3. in Theorem~\ref{thm::lat-free-2D}.


Let us make the notation uniform and use $H_k'$ and $\ell_k'$ denote the corresponding facet defining halfspace and hyperplane, and $v_j$ be the integer point located on the corresponding facet, for $j=1,2,3,4$.  By assumption, $\ell_k\neq \ell_k'$ for $k=2,3,4$. Let $\ell_5$, $\ell_5'$ and $\ell_6$ be the lines such that $\{v_1,v_2\}\subseteq \ell_5$, $\{v_3,v_4\}\subseteq \ell_5'$, and $\{v_1,v_4\}\subseteq \ell_6$. Furthermore, let $H_5$ be the halfspace defined by $\ell_5$ such that $v_4\notin H_5$. Similarly $H_6$ be the halfspace defined by $\ell_6$ such that $v_2\notin H_6$.  Since there are no integer points between $\ell_5$ and $\ell_5'$ (because $v_1, v_2, v_3, v_4$ form a fundmanetal parallelopiped of the integer lattice), all integer points in $H'_3 \cap H'_4$ are contained in $H_5$, except for the points $v_3$ and $v_4$. Since these points are not contained in $H_3\cap H_4$, we must have $(H_3\cap H_4)_I\subseteq H_5$. Similarly, we have $(H_2\cap H_3)_I\subseteq H_6$.
Therefore, $(H_2\cap H_3)_I\cap (H_4\cap H_3)_I\subseteq H_5\cap H_6\subseteq H$. \qed
\end{proof}

\begin{theorem} \label{thm::helly}(Integer Helly's Theorem~\cite{bell1977theorem,Doignon1973,scarf1977observation,hoffman1979binding})
Let $\mathcal{I}$ be a finite family of convex sets in $\R^n$ such that $\bigcap_{C\in \mathcal{I}}C\cap \Z^n=\emptyset$, then there exists $\mathcal{I}'\subseteq \mathcal{I}$ such that $\bigcap_{C\in \mathcal{I}'}C\cap \Z^n=\emptyset$ and $|\mathcal{I}'| \leq 2^n$.
\end{theorem}

\begin{corollary}\label{cor::one-facet-helly}
Given a full dimensional polyhedron $P\in\R^2$ with at least four facets, assume only one of its facets $F$ contains integer points in its relative interior. Let $H$ be the facet defining halfspace of $F$. Furthermore, we assume $\intt(H)\cap P\cap\Z^2=\emptyset$. Then there exists three facet defining halfspaces for $P$ denoted by $H_1$, $H_2$ and $H_3$ other than $H$, such that $\intt(H)\cap H_1\cap H_2\cap H_3\cap \Z^2=\emptyset$ and $H$ is irredundant to $\intt(H)\cap H_1\cap H_2\cap H_3$.  
\end{corollary}



\begin{theorem}\label{thm::non-empty-IH}
Given a polyhedron $P\in \R^2$ such that $P\cap \Z^2\neq \emptyset$. Then $P_I$ is the intersection of the integer hulls of all the simplicial cones (or split sets) constructed from pairs of facets of $P$.
\end{theorem}

\begin{proof}
Let $H$ be a halfspace containing $P_I$ and described by $\{x\in\R^2:\langle a,x \rangle\geq \delta\}$. We wish to show that $H$ is valid for the intersection of the integer hulls of all the simplicial cones (or split sets) constructed from pairs of facets of $P$. For this purpose, we may strengthen $H$ such that its bounding hyperplane has a nonempty intersection with $P_I$, and show that this strengthening has the desired property. Let $H'$ be the halfspace $\{x\in\R^2:\langle a,x \rangle\leq \delta\}$. Let $P$ be the intersection of halfspaces $H_i:=\{x\in \R^2:\langle a_i,x \rangle\leq \delta_i\}$, and the two entries of $a_i$ be coprime for $i=1,\dots,m$. If $\delta_i\in \Z$, then let $H_i'$ be the {halfspace} $\{x\in\R^2:\langle a_i,x \rangle\leq \delta_i+\frac{1}{2}\}$. Otherwise, $H_i'=H_i$. Let $P'$ be the intersection of $H_i'$ for $i=1,\ldots,m$. Note that $P_I=P'_I\subseteq\intt(P')$, and there is no integer point on any facet defining line of $P'$.

~

\noindent {\bf Claim}: $H$ is valid for the intersection of the integer hulls of all the simplicial cones (or split sets) constructed from pairs of facets of $P'$.
\begin{proof}
By our assumption, the bounding hyperplane of $H$ contains integer points from $P_I$. Thus, $H'\cap P_I\neq \emptyset$. Then $P'\cap H'$ is a lattice-free set with one facet defined by $H'$ and containing at least one integer point in its relative interior since $P'_I=P_I\subseteq \intt(P)$. If $P'\cap H'$ has only two or three facets, the proof is trivial. Otherwise, by Corollary  $\ref{cor::one-facet-helly}$, there exist three facet defining halfspaces of $P'$, say $H'_1$, $H'_2$ and $H'_3$, such that $Q:=H'_1\cap H'_2\cap H'_3\cap H'$ is lattice free and nonempty since $H'\cap P_I\neq \emptyset$, and $H'$ is irredundant to $Q$. Moreover, the facet of $Q$ defined by $H'$ contains a integer point in its relative interior since $P_I\subseteq \intt(P)$. 
If $Q$ only has two or three facets, then the proof is trivial. If $Q$ has four facets, then by Lemma \ref{lem::valid-ineq}, we can finish the proof.\qed
\end{proof}

The claim immediately implies that $P_I$ is the intersection of the integer hulls of all the simplicial cones (or split sets) constructed from pairs of facets of $P'$. The proof can be finished by the fact that $H_i\subseteq H_i'$ for $i=1,\ldots, m$.\qed
\end{proof}

\begin{theorem}\label{thm::empty-IH}
Given a polyhedron $P$ such that $P\cap \Z^2= \emptyset$, we have that the intersection, denoted by $U$, of the integer hulls of all the simplicial cones (or split sets) constructed from pairs of facets of $P$, is empty.
\end{theorem}

\begin{proof}
By Theorem \ref{thm::helly}, we can assume $P$ has at most four facets. If $P$ has two or three facets, then the proof is trivial. Therefore, assume $P$ has four facets and let $H_i$ for $i=1,\ldots,4$ denote the facet defining halfspaces in clockwise order. If $H_1\cap H_3$ or $H_2\cap H_4$ forms a split set, then the proof is trivial. So we assume both $H_1\cap H_3$ and $H_2\cap H_4$ contain integer points. Since $H_1\cap H_3 \cap \Z^2 \neq \emptyset$, $(H_1\cap H_3\cap H_2)_I$ or $(H_1\cap H_3\cap H_4)_I$ is not empty. Without loss of generality, assume $(H_1\cap H_3\cap H_2)_I\neq \emptyset$. By Theorem \ref{thm::non-empty-IH}, we have $(H_1\cap H_3\cap H_2)_I=(H_3\cap H_2)_I\cap (H_1\cap H_2)_I\cap (H_1\cap H_3)_I$. Therefore $U\subseteq (H_1\cap H_3\cap H_2)_I$. Similarly, using the fact that $H_2\cap H_4 \cap \Z^2 \neq \emptyset$, we can assume $(H_2\cap H_4\cap H_1)_I\neq \emptyset$ and have $U\subseteq (H_2\cap H_4\cap H_1)_I$. Hence $U\subseteq (H_1\cap H_3\cap H_2)_I \cap (H_2\cap H_4\cap H_1)_I\subseteq (H_1 \cap H_2\cap H_3)_I \cap H_4=\emptyset$.\qed
\end{proof}
\end{document}